\documentclass[11pt, reqno]{amsart}
\usepackage{amsmath, amssymb, amsthm}
\usepackage{geometry}
\usepackage{graphicx}
\usepackage{microtype}
\usepackage[hidelinks]{hyperref} 

\newtheorem{theorem}{Theorem}[section]

\newtheorem{proposition}[theorem]{Proposition}
\newtheorem{corollary}[theorem]{Corollary}

\theoremstyle{definition}

\newtheorem{example}[theorem]{Example}

\theoremstyle{remark}
\newtheorem{remark}[theorem]{Remark}
\title[Non-compact Ricci Solitons of FV with Potential Field of Constant Length]{\textbf{Non-compact Ricci Solitons of Finite Volume with Potential Field of Constant Length}}                                 
\author{Hemangi Madhusudan Shah, Sharief Deshmukh, and Mohammad Aqib}            
\date{}

\begin{document}

\begin{abstract}
We study non-compact Ricci solitons of finite volume whose potential
vector field has constant length. Under the assumptions that the scalar 
curvature is constant along the integral curves of the potential field 
and that a natural divergence term is integrable on the unit tangent bundle, we prove that such Ricci solitons are necessarily trivial. As applications, we characterise $3$-dimensional Ricci solitons of constant length and of finite volume on almost $\alpha$-cosymplectic manifolds. In fact, we show that these 
spaces are indeed compact quotients of Euclidean spaces.
We also study complete, non-compact and finite volume contact metric manifolds in dimension $3$ admitting a Reeb vector field as a Ricci soliton. We 
show that it is a compact spherical space form.
We include some examples to illustrate the necessity of the finite-volume and constant-length assumptions.
\end{abstract}

\maketitle

\section{Introduction}

Ricci flow, introduced by Hamilton \cite{hamilton1982three}, has become
one of the central tools in modern Riemannian geometry, most notably
through its role in Perelman's proof of the Poincar\'e and
geometrisation conjectures \cite{perelman2002entropy,perelman2003ricci}.
The self-similar solutions of the Ricci flow are known as \emph{Ricci
solitons} and provide the natural generalisation of Einstein metrics.
Formally, a Ricci soliton is a quadruple $(M^n, g, X, \lambda)$, where
$(M^n, g)$ is a Riemannian manifold, $X$ is a smooth vector field on
$M^n$, called the \emph{potential vector field}, $\lambda\in\mathbb{R}$
is called the \emph{soliton constant}, and the following equation holds:
\begin{equation}\label{e0}
\frac{1}{2}\mathcal{L}_X g+\operatorname{Ric}=\lambda g,
\end{equation}
where $\mathcal{L}_X g$ is the Lie derivative of the metric $g$ in the
direction of $X$ and $\operatorname{Ric}$ denotes the Ricci tensor of
$g$. A Ricci soliton is called \emph{shrinking}, \emph{steady}, or
\emph{expanding}, accordingly as $\lambda>0$, $\lambda=0$, or $\lambda<0$.
If $X$ is a Killing vector field, the Ricci soliton reduces to an
Einstein manifold, and in that case it is called \emph{trivial}. In the case when
$X=\nabla f$ for some smooth function $f$, the soliton is called a
\emph{gradient Ricci soliton}, and \eqref{e0} takes the form
\[
\operatorname{Hess} f + \operatorname{Ric} = \lambda g.
\]

Ricci solitons have been extensively studied under compactness or
completeness assumptions; see, for instance, \cite{cao2010recent,
chow2007ricci} for surveys and \cite{munteanu2013conical,
petersen2009rigidity} for structural results. The compact case is
comparatively well understood: for example, every compact steady Ricci soliton is necessarily trivial, and compact shrinking  
gradient Ricci solitons enjoy considerable rigidity. The situation in
the non-compact setting is much richer, and non-trivial examples such
as a Bryant soliton, a cigar soliton, and various homogeneous
solitons on Lie groups are known.

An intermediate class of particular geometric interest is that of 
\emph{non-compact, complete Ricci solitons of finite volume}. Such manifolds
arise naturally as, for instance, cusped hyperbolic manifolds, certain
warped products with an integrable warping function, and finite volume quotients of non-compact
homogeneous spaces by non-cocompact discrete subgroups.\\
\indent
To the best of our knowledge, the literature has studied finite-volume Ricci solitons sparingly.
In this article, we begin a study of finite-volume 
Ricci solitons. We show that a class of these solitons are trivial and compact. Finite-volume,
non-compact manifolds inherit many of the analytic advantages of the
compact case: in particular, a version of the divergence theorem
continues to hold. Melo and Pujals \cite{melo2016divergence}, using a
dynamical approach based on the Poincar\'e recurrence theorem, proved the following result:
\begin{theorem}[Corollary $1.2$, \cite{melo2016divergence}] \label{fvt}
Let $(M,g)$ be a complete, non-compact Riemannian manifold
without boundary and of finite volume. Then for 
a $C^1$ vector field $X$ on $M$, if 
$$f_{X}(p,v) = g \left(\nabla_{v} X, v \right)$$ is integrable on $SM$,
then $\operatorname{div} X$ is integrable on $M$ and 
$$\int \limits_{M}  \operatorname{div} X\; dV = 0.$$
\end{theorem}

This principle is the key analytic tool
that allows us to investigate rigidity arguments from the compact setting
to non-compact finite-volume manifolds, and it is used in an essential
way in the proof of Theorem \ref{main}.

The purpose of this paper is to study non-compact Ricci solitons of
finite volume whose potential vector field $X$ has constant length.
The constant length of the soliton condition is not artificial: it appears naturally in almost
contact geometry, where the Reeb vector field $\xi$ is by definition a
unit vector field. Consequently, our results have direct applications
to the Reeb-potential Ricci solitons studied for almost contact metric,
almost $\alpha$-cosymplectic, Kenmotsu, K-contact, Sasakian, and
trans-Sasakian manifolds; see \cite{blair2002riemannian,
perrone2018classification, li2022ricci} for background and some works
in these directions.\\
\indent
Our main abstract rigidity result (Theorem~\ref{main}) shows that, if
$X$ has constant length, the scalar curvature $S$ is constant along
the flow of $X$, and a natural function on the unit tangent bundle is
integrable, then the Ricci soliton is trivial. As applications of this to almost contact geometry, we obtain rigidity and
non-existence results for Ricci solitons with Reeb potential vector
field. In particular, Theorem~\ref{c2} shows that a non-compact, complete
almost $\alpha$-cosymplectic Ricci soliton of finite volume with Reeb
potential must be cosymplectic, steady, trivial, and flat; because of completeness,
it must be a compact Bieberbach quotient of $\mathbb{R}^3$.
Corollary \ref{class}
give a classification in dimension three of almost
$\alpha$-cosymplectic Ricci solitons with unit potential field. In the contact case, we prove
(Theorem~\ref{t:contact}) that any $3$-dimensional contact Ricci
soliton with Reeb potential is necessarily trivial, shrinking, and
Sasakian Einstein of constant sectional curvature $1$; when complete,
it is a compact spherical space form.\\
\indent
We also present examples 
that illustrate the sharpness of the various hypotheses.  The flat torus
$\mathbb{T}^3$ (Example~\ref{ex:torus}) realises the compact quotient as proven by Theorem~\ref{c2}; the $\mbox{Sol}_3$ example
(Example~\ref{ex:sol3}), shows that the
finite-volume hypothesis cannot be dropped; and the round sphere
$\mathbb{S}^3$ (Example~\ref{ex:sphere}) shows that the shrinking case in Theorem~\ref{t:contact} is indeed 
realized.

The paper is organised as follows. In Section~\ref{sec:preliminaries}, we derive the fundamental integral identity used throughout the paper,
and prove our main theorem about the triviality of Ricci solitons of constant 
length vector under the finite volume assumption.
In Section \ref{app}, we recall the basic facts on
almost contact and almost $\alpha$-cosymplectic manifolds. We establish a characterisation of $3$-dimensional almost $\alpha$-cosymplectic manifolds
admitting a finite length Ricci soliton of finite volume as: it is a compact quotient 
of Euclidean space. The Section \ref{sec:contact} is devoted to contact Ricci solitons in
dimension three. We establish that if its Reeb vector field is a Ricci soliton, then it is a compact spherical Ricci soliton.

\section{Main Theorem}\label{sec:preliminaries}
In this section, we prove the paper's main theorem, Theorem \ref{main}. We first analyse the condition of a constant-length potential vector field
and derive a fundamental identity which will be used throughout this article.

\subsection{Ricci solitons with constant-length potential field}
Let $\left(M^{n},g,X, \lambda \right) $ be an $n$-dimensional Ricci soliton.
We shall have
\begin{equation}\label{e1}
{\nabla} _{U}X=\lambda U-QU+\phi U\text{,\quad }U\in \Gamma \left(
TM^{n}\right) \text{,}  
\end{equation}
where $\lambda $ is a constant and $Q$ is the Ricci operator and $\phi$ is a skew-symmetric tensor related to $X$ through its dual $1$-form $\eta$ by
\[
\frac{1}{2}d\eta \left( U,V\right) =g\left( \phi U,V\right) \text{,\quad }%
U,V\in \Gamma \left( TM^{n}\right) \text{.}
\]
Suppose that $X$ has constant length. Then by virtue of equation (\ref{e1})
\[
0=U\left( \left\Vert X\right\Vert ^{2}\right) =2g\left( \lambda U-QU+\phi
U,X\right) \text{,\quad }U\in \Gamma \left( TM^{n}\right) \text{,}
\]
that is,
\begin{equation}\label{e2}
\lambda X-QX -\phi X=0\text{.}  
\end{equation}

Taking the inner product of \eqref{e2} with $X$ and using the
skew-symmetry of $\phi$, we obtain
\[
\operatorname{Ric}(X,X)=\lambda \|X\|^2.
\]
Thus, if $\|X\|^2=c>0$, then
\begin{equation}\label{e3}
\operatorname{Ric}\left( X,X\right) =c\lambda \text{,}  
\end{equation}
which is constant whenever $\|X\|$ is constant. Equation (\ref{e2}) will be one of the main tools in studying the geometry of such Ricci solitons.

\subsection{Fundamental Identities}

Differentiating equation (\ref{e2}), while using equation (\ref{e1}), we see
\begin{eqnarray*}
\lambda \left( \lambda U-QU+\phi U\right)  &=&\left( \nabla Q\right) \left(U,X\right) +Q\left( \lambda U-QU+\phi U\right)  \\ 
&&+\left( \nabla \phi \right) \left( U,X\right) +\phi \left( \lambda
U-QU+\phi U\right) \text{.}
\end{eqnarray*}
Choosing a local orthonormal frame $\left\{ u_{1},..,u_{n}\right\} $,
taking $U=u_{i}$ and on taking the inner product with $u_{i}$ and summing
the resulting equation, while keeping in mind the symmetry of $Q$ and skew  
symmetry of $\phi $, we land at
\begin{equation}\label{e4}
n\lambda ^{2}-\lambda S=\frac{1}{2}X(S)+\lambda S-\left\Vert Q\right\Vert
^{2}-\sum\limits_{i}g\left( X,\left( \nabla \phi \right) \left(
u_{i},u_{i}\right) \right) -\left\Vert \phi \right\Vert ^{2}\text{,}  
\end{equation}
where we have used the following
\[
\sum\limits_{i}g\left( Qu_{i},\phi u_{i}\right) =0\text{,\quad }
\sum\limits_{i}g\left( Qu_{i},Qu_{i}\right) =\left\Vert Q\right\Vert ^{2}
\text{,\quad }\sum\limits_{i}g\left( \phi u_{i},\phi u_{i}\right)
=\left\Vert \phi \right\Vert ^{2}\text{.}
\]
Rearranging equation (\ref{e4}), we get%
\begin{eqnarray}\label{e5}
\left\Vert Q\right\Vert ^{2}-\frac{1}{n}S^{2} &=&-\frac{1}{n}\left(
S-n\lambda \right) ^{2}+\frac{1}{2}X(S)   \\
&&-\sum\limits_{i}g\left( X,\left( \nabla \phi \right) \left(
u_{i},u_{i}\right) \right) -\left\Vert \phi \right\Vert ^{2}\text{.} 
\nonumber
\end{eqnarray}

Now, we use equation (\ref{e1}) in computing $\operatorname{div}\left(\phi X\right)$,
we have
\[
\operatorname{div}\left(\phi X \right) =\sum\limits_{i}g\left( \nabla _{u_{i}}\phi
X,u_{i}\right) =\sum\limits_{i}g\left( \left( \nabla \phi \right) \left(
u_{i},X\right) +\phi \left( \nabla _{u_{i}}X\right) ,u_{i}\right) \text{,} 
\]
that is,
\[
\operatorname{div}\left( \phi X\right) =\sum\limits_{i}g\left( \left( \nabla \phi
\right) \left( u_{i},X\right) +\phi \left( \lambda u_{i}-Qu_{i}+\phi
u_{i}\right) ,u_{i}\right) \text{.} 
\]
Using skew symmetry of $\phi $, above equation reduces to
\begin{eqnarray}\label{e6}
\operatorname{div}\left(\phi X \right) = -\sum \limits_{i} g\left(X,\left(\nabla \phi \right) \left( u_{i},u_{i}\right) \right) -\left\Vert \phi \right\Vert
^{2}\text{,}  
\end{eqnarray}
where we have used
\[\text{div} (
\sum\limits_{i}g\left( Qu_{i},\phi u_{i}\right) =0\text{.} 
\]
Inserting equation (\ref{e6}) in equation (\ref{e5}), we arrive at
\begin{equation}\label{e7}
\left\Vert Q\right\Vert ^{2}-\frac{1}{n}S^{2}=-\frac{1}{n}\left( S-n\lambda
\right) ^{2}+\frac{1}{2}X\left( S\right) +\operatorname{div}\left(\phi X \right) 
\text{.}  
\end{equation}
This identity is the main tool in the analysis of Ricci solitons of finite volume
under consideration. Now we are in a position to 
prove our main theorem.

\begin{theorem}\label{main}
Let $\left( M^{n},g,X,\lambda \right)$ be an $n$-dimensional non-compact, complete
Ricci soliton of finite volume without boundary with length of potential
field $X$ a constant. If the scalar curvature $S$ is constant along the
integral curves of $X$ and the function $f=\left\langle \nabla _{v}\phi
X,v\right\rangle $ is integrable on the unit tangent bundle $SM^{n}$ of $
M^{n}$, then $\left( M^{n},g,X,\lambda \right) $ is trivial, that is $X$ 
is a Killing vector field.
\end{theorem}
\begin{proof}
By hypothesis and by Theorem \ref{fvt}, it follows that 
\begin{eqnarray}\label{a}
\int_{M^n}\operatorname{div}\left(\phi X \right) dV = 0.
\end{eqnarray}
Therefore, from (\ref{e7}), we have that 
\begin{equation}\label{e8}
\left\Vert Q\right\Vert ^{2}-\frac{1}{n}S^{2}+\frac{1}{n}\left( S-n\lambda
\right)^{2}- \frac{1}{2}X\left( S\right) 
 \end{equation}
is integrable.

Now, integrating equation (\ref{e7}) and assuming that the scalar curvature $S$
is a constant along the integral curves of the potential field $X$, we
conclude 
\begin{equation}\label{e8a}
\int\limits_{M^{n}}\left( \left\Vert Q\right\Vert ^{2}-\frac{1}{n}
S^{2}\right) dV +\frac{1}{n}\int\limits_{M^{n}}\left(S-n\lambda \right)
^{2}dV = 0.
\end{equation}
Note that the Schwartz's inequality implies
\begin{equation}\label{e9}
\left\Vert Q\right\Vert ^{2}\geq \frac{1}{n}S^{2}.  
\end{equation}
Hence, both the brackets in the above equation are non-positive. Hence, we conclude
\begin{equation}\label{e10}
\int \limits_{M^{n}}\left(\left\Vert Q\right\Vert ^{2}-\frac{1}{n} S^{2}\right) dV=0\text{,\quad }\int\limits_{M^{n}}\left( S-n\lambda \right)
^{2}dV=0.  
\end{equation}
And it amounts to equality in the inequality (\ref{e9}), which holds if and only
if $Q=\frac{S}{n}I$ and the second equation in (\ref{e10}) implies $S=n\lambda $.
Thus, we conclude.
\[
Q=\lambda I\text{,} 
\]%
that is, $Ric=\lambda g$, and by the definition of a Ricci soliton, we get
\[
\pounds _{X}g=0 
\]
and consequently that $\left( M^{n},g,X,\lambda \right) $ is a trivial Ricci
soliton. 

\end{proof}

As an application of Theorem \ref{main} we obtain:

\begin{corollary}\label{trivial}
Let $\left(M^{n},g, X,\lambda \right)$ satisfy the hypothesis
of the above theorem and assume that it is connected, complete, and without conjugate points. Then
the soliton must be shrinking.
\end{corollary}
\begin{proof}
Let $h = \frac{||X||^2}{2}.$ Then, as $X$ is a constant-length Killing vector field, we have
$$ \Delta h = ||\nabla X ||^2 - \operatorname{Ricci}(X,X) = 0.$$
If $\operatorname{Ricci}(X, X) \leq 0$, then $X$ must be a
parallel vector field. By hypothesis on $M$,
every geodesic of $M$ is a line, this implies that 
$M = N \times \mathbb{R}.$  This is a contradiction to 
the fact that the soliton is of finite volume. Thus, Ricci$(X, X) > 0$, 
which implies that $\lambda > 0$.
\end{proof}

\begin{remark}
 Following the proof of Theorem \ref{main}, it is easy to see that the hypothesis $X(S)=0$ can be replaced by  weaker hypotheses: (i) $X(S)\le0$ and (ii) $X(S)\in L^1(M)$ and $\int_{M^n}X(S)\,dV\le0$. These hypotheses
 also yield that the soliton vector field under consideration is Killing.
\end{remark}

As an application, we obtain a rigidity result for gradient Ricci solitons. Recall that a Ricci soliton is called a gradient Ricci 
soliton if the potential field is the gradient of a smooth function, 
$X=\nabla f$, so that the soliton equation becomes
$\operatorname{Hess}f+\operatorname{Ric}=\lambda g$. Since the dual $1$-form $\eta=df$ satisfies $d\eta=0$, the skew-symmetric tensor
$\phi$ of equation \eqref{e1} vanishes identically. In particular,
$\operatorname{div}(\phi X)=0$ and the integrability condition on
$SM^{n}$ in Theorem \ref{main} is automatically satisfied.

\begin{corollary}\label{cor:gradient}
Let $(M^{n},g,\nabla f,\lambda)$ be a complete, non-compact, simply connected gradient
Ricci soliton of finite volume without boundary whose gradient \(\nabla f\) has positive constant length. If
$\nabla f(S)\leq 0$, then $(M^{n},g)$ is Einstein with
$\operatorname{Ric}=\lambda g$. In particular, the gradient Ricci soliton is trivial; that is, $f$ is constant.
\end{corollary}
\begin{proof}
By Theorem \ref{main}, we have that
$\mathcal{L}_{\nabla f}g= {\nabla}^2 f =0$.
This means that $\nabla f$ is a parallel
vector field and $\operatorname{Ric}=\lambda g$.
 Suppose, to the contrary, that $\nabla f$ is not identically zero.
Since $\nabla f$ is parallel, it has constant positive length and
defines a nowhere-vanishing parallel vector field on $M^{n}$. By the
de Rham decomposition theorem, as $M^{n}$ is connected and as $\nabla f$ is bounded, it has a globally defined flow. Therefore, it splits isometrically as
\[
M^{n}\cong N^{n-1}\times\mathbb{R}.
\]
But this contradicts the fact that $M$ has finite volume. Therefore $\nabla f$ must vanish at a
point and consequently at every point. Hence,
$f$ is constant, and the gradient Ricci soliton is
trivial.
\end{proof}

\begin{remark}
For a gradient Ricci soliton, the condition
\(\nabla f(S)\le0\) means that the scalar curvature is
non-increasing along the gradient flow lines of the
potential function \(f\). Thus the above corollary
 shows that even this weak monotonicity
assumption forces the soliton to be Einstein and, under the finite-volume hypothesis, trivial.
\end{remark}

\section{Applications of  Main Theorem to Contact Geometry} \label{app}
In this section, using our main theorem, we 
describe $3$-dimensional complete almost $\alpha$-cosymplectic manifolds admitting Ricci solitons
of finite volume and show that indeed they are compact quotients of Euclidean space. In what follows, we will be using the Bieberbach 
theorem which says that a complete, flat manifold 
of finite volume is necessarily compact.\\

\indent
An almost contact metric manifold $\left( M^{2n+1},g, \varphi,\xi ,\eta
\right) $ comprises a $(2n+1)$-dimensional Riemannian manifold $\left(
M^{2n+1},g\right)$, equipped with a $(1,1)$ tensor field $\varphi$, a unit
vector field $\xi $ called Reeb vector field and its dual $1$-form $\eta$
satisfying
\[
 {\varphi}^{2}=-I+\eta \otimes \xi \text{,\quad } \varphi \xi =0\text{,\quad }%
\eta \circ \varphi =0\text{,\quad }g\left( \varphi U, \varphi V\right)
=g\left( U,V\right) -\eta \left( U\right) \eta \left( V\right) \text{.}
\]
As the Reeb vector field $\xi $ of an almost contact metric manifold
$\left( M^{2n+1},g, \varphi ,\xi ,\eta \right) $ is a unit vector field, as a
direct consequence of the above Theorem \ref{main}, we have:

\medskip 

\begin{corollary}\label{c1}
Let $\left( M^{2n+1},g,\xi ,\lambda \right) $ be a $(2n+1)$-dimensional
non-compact Ricci soliton of finite volume without boundary with potential
field $\xi $, the Reeb vector field, of the almost contact metric manifold 
$\left( M^{2n+1},g, \varphi,\xi ,\eta \right)$ satisfying the hypothesis
of Theorem \ref{main}, then $\xi$ is a Killing vector field.
\end{corollary}

An almost contact metric manifold $\left( M^{2n+1},g,\varphi ,\xi ,\eta
\right) $ is said to be a $K$-contact manifold, if the Reeb vector field $\xi 
$ is Killing. Moreover, an almost contact metric manifold $\left(
M^{2n+1},g,\varphi ,\xi ,\eta \right) $ is said to be a Sasakian manifold if
it satisfies%
\[
\left( \nabla \varphi \right) \left( U,V\right) =\eta \left( V\right)
U-g\left( U,V\right) \xi. 
\]%
And it follows that a Sasakian manifold is a $K$-contact manifold, while the
converse is not true. In contact geometry, one of the interesting questions
is to find conditions under which an almost contact metric manifold is a $ K$-
-contact manifold and a similarly important question is to find conditions under
which a $K$-contact manifold is a Sasakian manifold. In this respect in \cite{boyer2001einstein},
authors have proved that a compact Einstein $K$-contact manifold is a
Sasakian Einstein manifold (cf. Theorem A in \cite{boyer2001einstein}).\\
Now, observe that by Corollary \ref{c1}, if the $(2n+1)$-dimensional non-compact
Ricci soliton $\left( M^{2n+1},g,\xi ,\lambda \right)$ of finite volume,
satisfying the hypothesis of Corollary \ref{c1},  then $\xi $ is Killing, that is $\left(
M^{2n+1},g, \varphi ,\xi ,\eta \right) $ is a $K$-contact Einstein manifold.
However, we can not apply the result in \cite{boyer2001einstein} to make it a Sasakian Einstein manifold owing
to the non-compactness of $M^{2n+1}$. Yet, we have the following:

\medskip 

Note that on an almost contact metric manifold $\phi \xi = 0$ and hence clearly  
$f=\left\langle \nabla _{v} \varphi \xi
,v\right\rangle \equiv 0 $ is integrable on the unit tangent bundle $SM^{2n+1}$ (see also Theorem \ref{t:contact}). 

\begin{proposition}
A $(2n+1)$-dimensional non-compact almost contact metric manifold $\left(
M^{2n+1},g, \varphi ,\xi ,\eta \right) $ of finite volume such that the Reeb
vector field $\xi$ gives rise to the Ricci soliton $\left( M^{2n+1},g,\xi
,\lambda \right) $ is a  $K$-contact Einstein manifold.
\end{proposition}

\indent An almost contact metric manifold 
$\left(M^{2n+1},g, \varphi, \xi ,\eta \right)$ 
is called as  \textit{almost $\alpha$-cosymplectic metric manifold}, if $d\eta=0$ and $d \Phi=2\alpha\eta\wedge \varphi$, for some constant $\alpha$, where $2$-form $\Phi$ is defined as $\Phi(U,V)=g(U, \varphi V)$. 
In case $\alpha=0$,  an almost $\alpha$-cosymplectic manifold is \textit{cosymplectic}. 
An $\alpha$-cosymplectic manifold is a {\it normal} almost $\alpha$-cosymplectic manifold. 
For more details on almost cosymplectic and cosymplectic manifolds, we refer to \cite{perrone2018classification}
and references therein.
 We observe that the term 
``almost coK\"{a}hler manifold" is referred to as
``almost cosymplectic manifold" in the article
\cite{perrone2018classification}.
Cosymplectic/almost cosymplectic manifolds are odd-dimensional equivalent concepts of K\"{a}hler/almost K\"{a}hler manifolds.

\smallskip
An $\alpha$-cosymplectic $3$-manifold
admits a $\varphi$-basis \cite{blair2002riemannian} as:
\begin{proposition}[p. 70, \cite{perrone2018classification}]\label{h} 
On almost $\alpha$-cosymplectic $3$-manifold
  $M^{3}(\varphi,\xi,\eta,g)$,  there exists an orthonormal $\varphi$-basis $\{e,\varphi e, \xi \}$ satisfying
  \begin{eqnarray*}
  he=\sigma e,~~~~
  h \varphi e=-\sigma \varphi e,~~~~
  h\xi=0,
  \end{eqnarray*}
  with $\sigma$ a local smooth eigenfunction of $h$.
\end{proposition}

We would require the following table of the Levi-Civita connection.

\begin{theorem}[p. 70, \cite{perrone2018classification}]\label{lc} 
  The Levi-Civita connection on almost $\alpha$-cosymplectic $3$-manifold is described by,
  \begin{equation}\label{eq5}
    \begin{cases}
       & \nabla_ee=-a \varphi e-\alpha\xi,~~~~ \nabla_{\varphi e} e=-b \varphi e+\sigma\xi,~~~~ \nabla_\xi e=\mu \varphi e,  \\
       & \nabla_e \varphi e=ae+\sigma\xi,~~~~ \nabla_{\varphi e} p e=be-\alpha\xi,~~~~ \nabla_\xi \varphi e=-\mu e,  \\
       & \nabla_e\xi =\alpha e-\sigma \varphi e,~~~~ \nabla_{\varphi e}\xi=-\sigma e+\alpha \varphi e,~~~~ \nabla_\xi\xi=0,
    \end{cases}
  \end{equation}
   where $a=g(\nabla_e \varphi e, e)$, $b=-g(\nabla_{\varphi e}e, \varphi e)$ and $\mu=g(\nabla_\xi e, \varphi e)$ are smooth functions.
\end{theorem}

\begin{proposition}\cite{perrone2018classification}\label{r}
The scalar curvature $r$ of almost $\alpha$-cosymplectic $3$-manifold is expressed by
\begin{equation}
r=-6 \alpha^{2}-\operatorname{tr} h^{2}-2\left(a^{2}+b^{2}\right)-2(\varphi e)(a)+2 e(b). 
\end{equation}
\end{proposition}

It is worth noting that, if an almost contact metric manifold $\left(
M^{2n+1},g, \varphi ,\xi ,\eta \right) $ is an $\alpha $-cosymplectic
manifold, then in this case the dual $1$-form $\eta $ to the Reeb vector
field $\xi $ is closed. Thus, for the Ricci soliton $\left( M^{2n+1},g,\xi
,\lambda \right) $, the operator $\phi $ appearing in equation (\ref{e1}) is zero
and accordingly the function $f=\left\langle \nabla _{v}\varphi \xi
,v\right\rangle =0$ is automatically integrable on the unit tangent bundle $SM^{2n+1}$ as required in Corollary \ref{c1}. Using this argument, we prove:


\begin{theorem}\label{c2}
Let $\left( M^3,g,\xi, \lambda \right)$ be a $3$-dimensional
non-compact, complete Ricci soliton of finite volume without boundary with the potential 
field $\xi$, the Reeb vector field of the almost 
$\alpha$-cosymplectic manifold
 $\left( M^{3},g, \varphi ,\xi ,\eta \right) $. If the scalar curvature $S$ 
is constant along the integral curvres of $\xi$, then $\left(M^{3},g,\xi,\lambda \right)$ is cosymplectic, trivial and steady Ricci soliton. Consequently,  $M^3$ is flat and  $M^3$ is either $T^3$ or $M^3 = {\mathbb R}^3/{\Gamma}$ is a compact steady trivial Ricci soliton. Here, ${\Gamma}$ is a torsion-free crystallographic group, a Bieberbach group. 
\end{theorem}
\begin{proof}
By Corollary \ref{c1}, the Reeb vector field 
$\xi$ is Killing, which implies that $\nabla \xi$
is skew-symmetric. But for almost $\alpha$ cosymplectic
manifold $\nabla \xi$ is symmetric by \cite{perrone2018classification}.
Hence,
 $\nabla \xi = 0$, which implies that $\xi$ is parallel 
 vector field. 
  The Levi-Civita connection on almost $\alpha$-cosymplectic $3$-manifold $M^{3}(\varphi,\xi,\eta,g)$ is described explicitly
 on p. 70, \cite{perrone2018classification}. We have
$$\nabla_e\xi =\alpha e-\sigma \varphi e = 0.$$
This implies that $\alpha = 0 = \sigma$.
Therefore, $M^3$ is cosymplectic and because we have
$$ Q\xi=-(2\alpha^2+\operatorname{tr}h^2)\xi+(2b\sigma-e(\sigma))\varphi e-(2a\sigma+(\varphi e)(\sigma))e,$$  
Ricci$(\xi, X) = 0$ for any $X$ and therefore,  $\lambda = 0$ and the soliton is  steady. \\
But then from the Ricci soliton equation we infer 
\begin{equation}
    g(\nabla_{U}\xi,V)+g(U,QV)= \lambda \; g(U,V) = 0,
\end{equation}
therefore,
\begin{equation}
    g(U,QV)= 0.
\end{equation}
Hence, $M^3$ is Ricci flat, and in turn it is flat. 
The final assertion follows by the Bieberbach theorem.
\end{proof}

Now we delve into almost $\alpha$-cosymplectic 
$3$-manifolds admitting transversal Ricci solitons.

\begin{theorem}\label{tr1}
Let $\left(M^3, g, X, \lambda \right)$ be a $3$-dimensional, complete, non-compact almost  $\alpha$-cosymplectic 
Ricci soliton of finite volume without boundary with the potential
field $X$, a unit vector field, transversal to $\xi$, the Reeb vector field. Then $X$ is a trivial Ricci soliton if the scalar curvature $S$ 
is constant along the integral curves of $X$,
$h, ||\nabla h||$ and $||\nabla X||$ are  bounded functions on $M$. In this case, $M$ is cosymplectic and $M^3$ is either $T^3$ or $M^3 = {\mathbb R}^3/{\Gamma}$ is compact steady trivial Ricci soliton. Here, ${\Gamma}$ is a torsion-free crystallographic group, a Bieberbach group. 
\end{theorem}
\begin{proof}
Let $M^3$ be almost $\alpha$-cosymplectic manifold, as $X$ is a transversal unit vector field, we have 
$X = f_1 e + f_2  \varphi e$, where ${f_1}^2 + {f_2}^2 =1$.  Hence, clearly $f_1, f_2$ are bounded. 
Here, $\eta_{1}(.) = g(X, .)$ and by (\ref{e2}),
$${\phi}_{1} X = \lambda X- QX \text{.}$$  
Therefore, by hypothesis, it follows from
Proposition \ref{h} and Theorem \ref{lc} 
that 
\begin{equation}\label{l4}
g\left({\nabla}_{v}{\phi}_{1} X, v \right)
\end{equation}
is integrable on $SM$.
Hence, by Theorem \ref{main}, $X$ is a trivial 
Ricci soliton. This implies that 
Ricci $=\lambda g$ with $\lambda = \frac{r}{3}$.
But then $r$ is constant, and consequently its sectional 
curvatures are all constants. 
By \cite{perrone2018classification}, 
$K(e, \phi e) = (r/2) + \mbox{tr} h^2 + 2 {\alpha}^2,$ 
is constant. This implies that $\sigma$ is constant.
and thus $h = 0$ and $\alpha = 0$. This shows that
$M^3$ is  cosymplectic manifold hence $Q(\xi, \xi) =0 = r$. Finally, we obtain that $M^3$ is Ricci flat
and hence flat; therefore $\lambda = 0$. 
By the Bieberbach theorem, the complete finite volume quotient of a flat manifold is either $T^3$ or a quotient of a torsion-free crystallographic group $\Gamma$, a Bieberbach group. 
\end{proof}

From Theorem \ref{c2} and Theorem \ref{tr1}, 
we obtain a complete characterisation of 
Ricci solitons on non-compact, complete Riemannian 
manifolds of finite volume.    

\begin{corollary}\label{class}
Let  $\left(M^3, g, V, \lambda \right)$ be a $3$-dimensional almost $\alpha$-cosymplectic 
non-compact, complete Ricci soliton of finite volume without boundary with the potential field $V$, a unit vector field. Suppose that the scalar curvature $S$
is constant along the integral curves of $V$,
then $M^3$ is a compact steady trivial Ricci soliton. 
Moreover, $M^3$ is flat torus $T^3$  or $M^3 = {\mathbb R}^3/{\Gamma},$ where  $\Gamma$ is 
a torsion-free crystallographic group $\Gamma$, a Bieberbach group, provided: 
(i) $V = \xi$, a Reeb vector field or
(ii) $V = f_1 e + f_2 \phi e$, where
$h, ||\nabla h||, ||V||, ||\nabla V||$ 
  are bounded functions on manifolds.
\end{corollary}

\noindent
{\bf Note:} The following comparison will be followed
while referring to \cite{li2022ricci}. 
The soliton constant denoted $S$ in \cite{li2022ricci} is our
$\lambda$; the eigenvalue of the tensor $h$ denoted $\lambda$ in
\cite{li2022ricci} is our $\sigma$; the scalar curvature denoted $r$
in \cite{li2022ricci} is our $S$. The $\varphi$-basis
$\{\xi, e, \varphi e\}$ and the structure functions $a, b, \mu$
coincide with ours, since both papers use the framework of Perrone
\cite{perrone2018classification}.

\begin{example}\label{ex:torus}
By Theorem 3.3 of \cite{li2022ricci}, the abelian Lie group
$\mathbb{R}^{3}$ (as well as the group $\widetilde{E}^{2}$), equipped
with a flat left-invariant \emph{cosymplectic} structure, carries a
contact Ricci soliton $\left(\mathbb{R}^{3},g,\xi,0\right)$: the Reeb
field $\xi$ is parallel, hence a unit Killing field, $\operatorname{Ric}=0$, and the
soliton is steady and trivial, in agreement with Corollary 3.4 of
\cite{li2022ricci} and with our Theorem \ref{c2}. Passing to the
quotient by the integer lattice $\mathbb{Z}^{3}$, whose translations
preserve the left-invariant structure, we obtain the flat torus
\[
\mathbb{T}^{3}=\mathbb{R}^{3}/\mathbb{Z}^{3},
\]
a finite-volume cosymplectic manifold carrying the steady trivial
contact Ricci soliton and realising precisely the conclusion
$M^{3}=\mathbb{R}^{3}/\Gamma$ of Theorem \ref{c2}. All hypotheses of Theorem \ref{c2} are satisfied: $\left\Vert\xi\right\Vert=1$, the
scalar curvature $S=0$ is constant, and since $d\eta=0$ the operator
$\phi$ of equation (\ref{e1}) vanishes, so that the integrability
condition on the unit tangent bundle holds automatically.
\end{example}

The following example shows that the finite volume and constant length hypothesis cannot be removed from our results.

\begin{example}\label{ex:sol3}

Following Theorem 4.1(i) of \cite{li2022ricci}, consider the
unimodular Lie group
\[
\mbox{Sol}_{3}=\mathbb{R}^{2}\rtimes_{A}\mathbb{R},
\qquad
A=\begin{pmatrix} 0 & -\sigma\\ -\sigma & 0\end{pmatrix},
\]
equipped with its standard left-invariant \emph{strict almost
cosymplectic} structure, that is, $\alpha=0$, $h\neq 0$, $\sigma$ a
non-zero constant, and $a=b=\mu=0$. By \cite{li2022ricci}, $\mbox{Sol}_{3}$
carries a Ricci soliton $\left(\mbox{Sol}_{3},g,V,\lambda\right)$ whose
potential field $V=f_{1}e+f_{2}\varphi e$ is \emph{transversal} to
$\xi$, with $\lambda=-2\sigma^{2}<0$,
that is, an expanding, non-trivial Ricci soliton; indeed $\mbox{Sol}_{3}$ is
homogeneous but not Einstein. It is simply connected,
non-compact, and of \emph{infinite volume}; the scalar curvature
$S=-2\sigma^{2}$ is constant, in particular constant along the
integral curves of $V$; and $a,b,\mu,\sigma$ are constants. It is instructive to observe that the potential field $V$ cannot have constant length. Indeed, specialising equations (4.3)--(4.7) of
\cite{li2022ricci} to $\mbox{Sol}_{3}$ (where $\alpha=a=b=\mu=0$) yields
\[
e(f_{1})=-2\sigma^{2},\qquad
\varphi e(f_{2})=-2\sigma^{2},\qquad
\varphi e(f_{1})+e(f_{2})=0 .
\]
Suppose $f_{1}^{2}+f_{2}^{2}$ were constant. Differentiating along $e$
and $\varphi e$ gives
\[
-4\sigma^{2}f_{1}+2f_{2}\,e(f_{2})=0,
\qquad
2f_{1}\,\varphi e(f_{1})-4\sigma^{2}f_{2}=0 .
\]
Substituting $\varphi e(f_{1})=-e(f_{2})$ into the second identity and
eliminating $e(f_{2})$ leads to $f_{1}^{2}+f_{2}^{2}=0$, that is,
$V=0$, which would force $\mbox{Sol}_{3}$ to be Einstein --- a contradiction.
Therefore $\left\Vert V\right\Vert$ is necessarily non-constant on $\mbox{Sol}_{3}$.  
\end{example}
  
\section{Contact Ricci solitons in dimension three}\label{sec:contact}
This section is devoted to the study of contact, complete and finite volume
Reeb Ricci solitons in dimension $3$.  We characterise such solitons completely and show that 
it is a compact quotient of a spherical space form.

\begin{theorem}\label{t:contact}
Let $(M^3, g, \xi, \lambda)$ be a $3$-dimensional complete, non-compact and finite volume contact Ricci
soliton with the Reeb potential field $\xi$ of the
contact metric manifold $(M^3, g, \varphi, \xi, \eta)$. Then
$\lambda = 2$, the tensor $h$ vanishes, and the soliton is trivial.
Moreover, $M^3$ is a Sasakian Einstein manifold of constant sectional
curvature $1$. In this case, $M^3=\mathbb{S}^3/\Gamma$ for a discrete group $\Gamma$ acting freely and
isometrically, and hence $M^3$ is compact.
In particular, there exists no steady or expanding
$3$-dimensional contact Reeb Ricci soliton with potential Reeb vector field. 
\end{theorem}
\begin{proof}
On the contact metric manifold $(M^3, g, \varphi, \xi, \eta)$ we have,
in the convention of \cite{blair2002riemannian},
$d\eta(U,V) = g(U, \varphi V)$, $U, V \in \chi(M^3)$, where the
exterior derivative carries the factor $\frac{1}{2}$. With respect to
the convention used in equation (\ref{e1}), namely
$d\eta(U,V) = U\eta(V) - V\eta(U) - \eta([U,V])$, the contact
condition reads $d\eta = 2\Phi$, and therefore the skew-symmetric
operator $\phi$ of equation (\ref{e1}) is identified as
$\phi = -\varphi$. (This is confirmed directly on $\mathbb{S}^3$,
where $\nabla_U \xi = -\varphi U$, $Q = 2I$ and $\lambda = 2$.) Because of contact geometry 
$\phi \xi = 0$. Hence, clearly $f$ is integrable and $\xi$ is a Killing vector field, by our main Theorem \ref{main}.

Equation (\ref{e1}) now reads
\begin{eqnarray}\label{m3}
{\nabla}_{X} \xi = \lambda X - QX - \varphi X.
\end{eqnarray}
Thus we have,
\begin{eqnarray}\label{m4}
R(X,Y)\xi = - (\nabla Q)(X,Y) + (\nabla Q)(Y, X)
- (\nabla \varphi) (X,Y) + (\nabla \varphi) (Y,X),
\end{eqnarray}
which, on contraction, yields
\begin{eqnarray}\label{m5}
\operatorname{Ric}(Y, \xi) = \frac{1}{2}Y(S) + g\left(Y, {\sum}_{i}
(\nabla \varphi)(e_i, e_i)\right),
\end{eqnarray}
that is,
\begin{eqnarray}\label{m6}
Q(\xi) = \frac{1}{2} \nabla S + \sum_{i}(\nabla \varphi )(e_i, e_i).
\end{eqnarray}
Choose an adapted frame $\{e, \varphi e, \xi\}$; we have the
following structure equations:
\begin{equation}\label{m7}
    \begin{cases}
       & \nabla_ee= a \varphi e +b \xi,~~~~ \nabla_{\varphi e} e=c \varphi e+ d\xi,~~~~ \nabla_\xi e= f \varphi e + s \xi,  \\
       & \nabla_e \varphi e= -ce+\mu\xi,~~~~ \nabla_{\varphi e} \varphi e= -ce +\nu\xi,~~~~ \nabla_\xi \varphi e=-fe + h \xi ,  \\
       & \nabla_e\xi =- b e - \mu \varphi e,~~~~ \nabla_{\varphi e}\xi=-d e - \nu \varphi e ,~~~~ \nabla_\xi\xi= -se - h \varphi e,
    \end{cases}
  \end{equation}
where $a, b, c, \cdots, h$ are smooth functions. Using
(\ref{m7}) in (\ref{m3}), we arrive at
\begin{equation}\label{m8}
    \begin{cases}
       & Q(e) = (\lambda+b)e + (\mu - 1) \varphi e,   \\
       & Q(\varphi e) = (d + 1)e +(\lambda + \nu)\varphi e,  \\
       & Q(\xi) = se + h \varphi e + \lambda \xi.
    \end{cases}
\end{equation}
On using symmetry of $Q$ in equation (\ref{m8}), one
confirms $s=0$ and $h = 0$, and in particular we have
\begin{equation}\label{m9}
Q(\xi) = \lambda \xi
\end{equation}
and $\mu - 1 = d + 1$ gives
\begin{equation}\label{m10}
\mu - d = 2.
\end{equation}
Now using equation (\ref{m7}) with $s = h = 0$, we compute
\begin{eqnarray}\label{m11}
\sum_{i} (\nabla \varphi)(e_i, e_i) & = &
(\nabla \varphi)(e, e) + (\nabla \varphi)(\varphi e, \varphi e) + (\nabla \varphi)(\xi, \xi)\\ \label{sumphi}
& = & (\mu - d) \xi = 2\xi,
\end{eqnarray}
by virtue of (\ref{m10}). Hence, by plugging in equations
(\ref{m6}), (\ref{m9}) and (\ref{sumphi}), we conclude
\begin{equation}\label{gradS}
\nabla S = 2(\lambda - 2) \xi.
\end{equation}
Taking the covariant derivative and using (\ref{m3}) we get
\begin{eqnarray}
H_S(X) = 2(\lambda - 2)(\lambda X - QX - \varphi X),
\end{eqnarray}
that is,
$$ \left(H_{S} - 2\lambda (\lambda - 2)I + 2(\lambda - 2) Q \right) X
= -2 (\lambda - 2) \varphi X.$$
The L.H.S. is symmetric, whereas the R.H.S. is skew-symmetric, which
implies that
$$ (\lambda - 2) \varphi X = 0, \quad X \in \chi(M^3).$$
Since $\varphi \neq 0$ on a contact manifold, we conclude
$\lambda = 2$.

Now, by (\ref{m9}), $\operatorname{Ric}(\xi,\xi) = \lambda = 2$. On a contact metric
$3$-manifold, one has
\[
\operatorname{Ric}(\xi,\xi)=2-\operatorname{tr}h^2
\]
\cite{blair2002riemannian}. Hence $\operatorname{tr}h^2=0$, that is,
$h=0$. Therefore $M^3$ is $K$-contact. In dimension three, every
$K$-contact manifold is Sasakian. Since $\xi$ is Killing, we have
$\mathcal{L}_\xi g=0$, and the soliton equation (\ref{e0}) gives
$\operatorname{Ric}=2g$. Hence, the soliton is trivial, and $M^3$ is Einstein. Since a
$3$-dimensional Einstein manifold has constant sectional curvature,
$M^3$ has constant sectional curvature $1$.Thus,
$M^3=\mathbb{S}^3/\Gamma$ for $\Gamma$ discrete
group acting freely and isometrically. In particular, $M^3$ is compact.
\end{proof}

The following example shows that the condition of the contact soliton can't be dropped in Theorem \ref{t:contact}. 

\begin{example}\label{ex:sphere}
Consider the unit sphere $\mathbb{S}^{3}(1)$ with its standard
Sasakian structure $\left(\varphi,\xi,\eta,g\right)$ arising from the
Hopf fibration. Then $\xi$ is a unit Killing field and $\operatorname{Ric}=2g$, so
that
\[
\frac{1}{2}\mathcal{L}_{\xi}g+\operatorname{Ric} = 2g,
\]
that is, $\left(\mathbb{S}^{3},g,\xi,2\right)$ is a shrinking contact
Ricci soliton with potential field the Reeb vector field; compact
quotients $\mathbb{S}^{3}/\Gamma$ furnish further examples. 
 On
$\mathbb{S}^{3}$ one checks directly that $\nabla_{U}\xi=-\varphi U$,
so that comparing with equation (\ref{e1}) (where $Q=2I$ and
$\lambda=2$) identifies the skew-symmetric operator as
$\phi=-\varphi$, and the structure functions of the adapted frame
\eqref{m7} satisfy $\mu=1$, $d=-1$. This shows that the non-trivial shrinking, expanding or steady contact solitons may be of infinite volume; and
might exist with the transversal potential field; that is, if the
 potential field is of the type $V = (f_1 e + f_2 \phi e) \in {\xi}^{\perp}$. 
Also, it should be noted that the contact condition $d\eta=2\Phi$
in Theorem \ref{t:contact} cannot be dropped: the flat torus $\mathbb{T}^3$ with its cosymplectic
structure carries the Reeb steady soliton $(\mathbb{T}^3,g,\xi,0)$
(it is non-shrinking).
\end{example}

\section*{Concluding Remarks}

The results obtained above show that the finite-volume assumption imposes
strong rigidity on Ricci solitons whose potential field has constant
length. In particular, in the almost contact setting, Reeb-potential
Ricci solitons are forced to be trivial under natural scalar curvature
and integrability assumptions. The examples showcased that the finite-volume hypothesis and the constant-length condition are essential.
In view of the result obtained, it would be interesting to investigate whether the integrability
assumption on $\operatorname{div}(\phi X)$ can be weakened, or whether similar rigidity phenomena hold for Ricci almost solitons and other generalised soliton structures.

\section{Declarations}
	
	\begin{itemize}
		\item Funding: Not applicable. 
		
		\item Conflict of interest/Competing interests: The authors have no conflict of interest and no financial interests in this article.
		
		\item Ethics approval: The submitted work is original and not submitted to more than one journal for simultaneous consideration.
		
		\item Consent to participate: Not applicable.
		
		\item Consent for publication: Not applicable.
		
		\item Availability of data and materials: This manuscript has no associated data.
		
		\item Code availability: Not applicable.
		
		\item Authors' contributions: Conceptualisation, methodology, investigation, validation, writing-original draft, review, editing, and heading have been performed by all the paper's authors.

        \item {\bf The authors declare that AI was neither used in any part of the arguments nor was it used in preparation or writing of this article.}
		
	\end{itemize}

\bibliographystyle{alpha}
\bibliography{ref}

\end{document}